\documentclass[a4paper,11pt,reqno]{amsart}

\usepackage[utf8]{inputenc}
\usepackage[T1]{fontenc}
\usepackage{amsfonts,amssymb}
\usepackage{graphicx}
\usepackage{array}
\usepackage{enumerate}
\usepackage{url}
\usepackage{tikz}
\usetikzlibrary{arrows,shapes,positioning,calc,3d,decorations.pathreplacing,decorations.markings,cd,patterns}
\tikzstyle{dot}=[shape=circle,draw,thick,color=black,minimum size=10pt,inner sep=.5pt]
\tikzstyle{->-}=[postaction={decorate,decoration={markings,mark=at position .5 with {\arrow{stealth}}}}]
\tikzstyle{-<-}=[postaction={decorate,decoration={markings,mark=at position .5 with {\arrowreversed{stealth};}}}]

\usepackage{xargs}
\usepackage{stmaryrd}
\usepackage{hyperref}
\usepackage{multirow}
\usepackage{ifthen}

\usepackage{empheq}
\numberwithin{equation}{section}

\usepackage{amsthm}
\theoremstyle{plain}
\newtheorem*{maintheorem}{Theorem}

\newtheorem{theorem}{Theorem}[section]
\newtheorem{lemma}[theorem]{Lemma}

\newtheorem{proposition}[theorem]{Proposition}

\theoremstyle{definition}

\theoremstyle{remark}
\newtheorem{remark}[theorem]{Remark}

\newcommand{\R}{\mathbb{R}}

\newcommand{\Q}{\mathbb{Q}}
\newcommand{\Z}{\mathbb{Z}}

\newcommand{\defeq}{\mathrel{\mathop{:}}=}

\DeclareFontFamily{U}{mathb}{}
\DeclareFontShape{U}{mathb}{m}{n}{ <5> <6> <7> <8> <9> <10> <12> gen * mathb <11> mathb10}{}
\DeclareSymbolFont{mathb}{U}{mathb}{m}{n}
\DeclareMathSymbol{\Asterisk}     {2}{mathb}{"06}
\newcommand{\bigast}{\mathop{\Asterisk}}
\newcommand{\match}{\nearrow}
\newcommand{\unmatch}{\searrow}

\makeatletter
\newcommand*{\addFileDependency}[1]{
   \typeout{(#1)}
   \@addtofilelist{#1}
   \IfFileExists{#1}{}{\typeout{No file #1.}}
}
\newcommand*{\myexternaldocument}[1]{%
   \addFileDependency{#1.tex}%
   \addFileDependency{#1.pdf}%
}

\myexternaldocument{figs} 
\makeatother

\newcommand{\VR}{\operatorname{VR}}
\newcommand{\diam}{\operatorname{diam}}
\newcommand{\Aut}{\operatorname{Aut}}

\renewcommand{\setminus}{\smallsetminus}

\newcolumntype{C}{>{$}c<{$}}

\newcommand\definevar[2]{%
  \expandafter\newcommand\csname var#1var\endcsname{#2}%
}
\newcommand\redefinevar[2]{%
  \expandafter\renewcommand\csname var#1var\endcsname{#2}%
}
\newcommand{\var}[1]{\csname var#1var\endcsname}

\newcommandx{\test}[2][1={}]{#1#2}

\newcommandx{\tinysimplex}[9][1={0},2={0},3={0},4={0},5={0},6={0},7={0}]{%
\begin{tikzpicture}[scale=.5,baseline={([yshift=-.8ex]current bounding box.center)}]
\def\r{.5}
\def\s{1.2}
\def\ang{50}
\definecolor{color7}{rgb}{1,1,1}
\definecolor{color9}{rgb}{1,1,1}
\definecolor{color12}{rgb}{1,1,1}
\definecolor{color13}{rgb}{1,1,1}
\definecolor{color14}{rgb}{1,1,1}
\definecolor{color15}{rgb}{1,1,1}
\definecolor{color#1}{rgb}{0,0,1}
\definecolor{color#2}{rgb}{0,0,1}
\definecolor{color#3}{rgb}{0,0,1}
\definecolor{color#4}{rgb}{0,0,1}
\definecolor{color#5}{rgb}{0,0,1}
\definecolor{color#6}{rgb}{0,0,1}
\definecolor{color#7}{rgb}{0,0,1}
\definevar{0}{}
\definevar{7}{}
\definevar{9}{}
\definevar{12}{}
\definevar{13}{}
\definevar{14}{}
\definevar{15}{}
\ifthenelse{#8>0}{
\redefinevar{#1}{1}
\redefinevar{#2}{2}
\redefinevar{#3}{3}
\redefinevar{#4}{4}
\redefinevar{#5}{5}
\redefinevar{#6}{6}
\redefinevar{#7}{7}
}{}
\node[dot,fill=color13,text=white] (13) at (180:\r) {\tiny\bfseries\var{13}};
\node[dot,fill=color14,text=white] (14) at (0:\r) {\tiny\bfseries\var{14}};
\node[dot,fill=color15,text=white] (15) at (\ang:\s) {\tiny\bfseries\var{15}};
\node[dot,fill=color7,text=white] (7) at(-\ang:\s) {\tiny\bfseries\var{7}};
\node[dot,fill=color12,text=white] (12) at (180-\ang:\s) {\tiny\bfseries\var{12}};
\node[dot,fill=color9,text=white] (9) at (180+\ang:\s) {\tiny\bfseries\var{9}};
\draw[thick] (14) -- (7) (14) -- (15)  (13) -- (9) (13) -- (12);
\draw[thick,#9] (13) -- (14);
\end{tikzpicture}
}
\newcommandx{\smallersimplex}[8][1={0},2={0},3={0},4={0},5={0},6={0},7={0}]{%
\begin{tikzpicture}[scale=.5,rotate=30,baseline={([yshift=-.8ex]current bounding box.center)}]
\def\r{1}
\def\s{1.5}
\def\ang{20}
\definecolor{color7}{rgb}{1,1,1}
\definecolor{color12}{rgb}{1,1,1}
\definecolor{color13}{rgb}{1,1,1}
\definecolor{color14}{rgb}{1,1,1}
\definecolor{color15}{rgb}{1,1,1}
\definecolor{color16}{rgb}{1,1,1}
\definecolor{color#1}{rgb}{0,0,1}
\definecolor{color#2}{rgb}{0,0,1}
\definecolor{color#3}{rgb}{0,0,1}
\definecolor{color#4}{rgb}{0,0,1}
\definecolor{color#5}{rgb}{0,0,1}
\definecolor{color#6}{rgb}{0,0,1}
\definecolor{color#7}{rgb}{0,0,1}
\definevar{0}{}
\definevar{7}{}
\definevar{12}{}
\definevar{13}{}
\definevar{14}{}
\definevar{15}{}
\definevar{16}{}
\ifthenelse{#8>0}{
\redefinevar{#1}{1}
\redefinevar{#2}{2}
\redefinevar{#3}{3}
\redefinevar{#4}{4}
\redefinevar{#5}{5}
\redefinevar{#6}{6}
\redefinevar{#7}{7}
}{}
\node[dot,fill=color14,text=white] (14) at (0,0) {\tiny\bfseries\var{14}};
\node[dot,fill=color15,text=white] (15) at (0:1) {\tiny\bfseries\var{15}};
\node[dot,fill=color13,text=white] (13) at (120:\r) {\tiny\bfseries\var{13}};
\node[dot,fill=color7,text=white] (7) at(-120:\r) {\tiny\bfseries\var{7}};
\node[dot,fill=color12,text=white] (12) at (60+\ang:\s) {\tiny\bfseries\var{12}};
\node[dot,fill=color16,text=white] (16) at (60-\ang:\s) {\tiny\bfseries\var{16}};
\draw[thick] (14) -- (7) (14) -- (15) (14) -- (13)
  (13) -- (12) (15) -- (16)
  (12) -- (16);
\end{tikzpicture}
}
\newcommandx{\smallsimplex}[9][1={0},2={0},3={0},4={0},5={0},6={0},7={0}]{%
\begin{tikzpicture}[scale=.5,baseline={([yshift=-.8ex]current bounding box.center)}]
\def\r{.5}
\def\s{1.2}
\def\t{1.5}
\def\ang{50}
\definecolor{color7}{rgb}{1,1,1}
\definecolor{color9}{rgb}{1,1,1}
\definecolor{color12}{rgb}{1,1,1}
\definecolor{color13}{rgb}{1,1,1}
\definecolor{color14}{rgb}{1,1,1}
\definecolor{color15}{rgb}{1,1,1}
\definecolor{color16}{rgb}{1,1,1}
\definecolor{color#1}{rgb}{0,0,1}
\definecolor{color#2}{rgb}{0,0,1}
\definecolor{color#3}{rgb}{0,0,1}
\definecolor{color#4}{rgb}{0,0,1}
\definecolor{color#5}{rgb}{0,0,1}
\definecolor{color#6}{rgb}{0,0,1}
\definecolor{color#7}{rgb}{0,0,1}
\definevar{0}{}
\definevar{7}{}
\definevar{9}{}
\definevar{12}{}
\definevar{13}{}
\definevar{14}{}
\definevar{15}{}
\definevar{16}{}
\ifthenelse{#8>0}{
\redefinevar{#1}{1}
\redefinevar{#2}{2}
\redefinevar{#3}{3}
\redefinevar{#4}{4}
\redefinevar{#5}{5}
\redefinevar{#6}{6}
\redefinevar{#7}{7}
}{}
\node[dot,fill=color13,text=white] (13) at (180:\r) {\tiny\bfseries\var{13}};
\node[dot,fill=color14,text=white] (14) at (0:\r) {\tiny\bfseries\var{14}};
\node[dot,fill=color15,text=white] (15) at (\ang:\s) {\tiny\bfseries\var{15}};
\node[dot,fill=color7,text=white] (7) at(-\ang:\s) {\tiny\bfseries\var{7}};
\node[dot,fill=color12,text=white] (12) at (180-\ang:\s) {\tiny\bfseries\var{12}};
\node[dot,fill=color16,text=white] (16) at (90:\t) {\tiny\bfseries\var{16}};
\node[dot,fill=color9,text=white] (9) at (180+\ang:\s) {\tiny\bfseries\var{9}};
\draw[thick] (14) -- (7) (14) -- (15)
  (13) -- (9) (13) -- (12) (15) -- (16)
  (12) -- (16);
\draw[thick,#9] (13) -- (14);
\end{tikzpicture}
}
\newcommandx{\simplex}[8][1={0},2={0},3={0},4={0},5={0},6={0},7={0}]{%
\begin{tikzpicture}[scale=.5,rotate=30,baseline={([yshift=-.8ex]current bounding box.center)}]
\def\r{1}
\def\s{1.5}
\def\ang{20}
\definecolor{color5}{rgb}{1,1,1}
\definecolor{color6}{rgb}{1,1,1}
\definecolor{color7}{rgb}{1,1,1}
\definecolor{color8}{rgb}{1,1,1}
\definecolor{color9}{rgb}{1,1,1}
\definecolor{color12}{rgb}{1,1,1}
\definecolor{color13}{rgb}{1,1,1}
\definecolor{color14}{rgb}{1,1,1}
\definecolor{color15}{rgb}{1,1,1}
\definecolor{color16}{rgb}{1,1,1}
\definecolor{color#1}{rgb}{0,0,1}
\definecolor{color#2}{rgb}{0,0,1}
\definecolor{color#3}{rgb}{0,0,1}
\definecolor{color#4}{rgb}{0,0,1}
\definecolor{color#5}{rgb}{0,0,1}
\definecolor{color#6}{rgb}{0,0,1}
\definecolor{color#7}{rgb}{0,0,1}
\definevar{0}{}
\definevar{5}{}
\definevar{6}{}
\definevar{7}{}
\definevar{8}{}
\definevar{9}{}
\definevar{12}{}
\definevar{13}{}
\definevar{14}{}
\definevar{15}{}
\definevar{16}{}
\ifthenelse{#8>0}{
\redefinevar{#1}{1}
\redefinevar{#2}{2}
\redefinevar{#3}{3}
\redefinevar{#4}{4}
\redefinevar{#5}{5}
\redefinevar{#6}{6}
\redefinevar{#7}{7}
}{}
\node[dot,fill=color14,text=white] (14) at (0,0) {\tiny\bfseries\var{14}};
\node[dot,fill=color15,text=white] (15) at (0:\r) {\tiny\bfseries\var{15}};
\node[dot,fill=color13,text=white] (13) at (120:\r) {\tiny\bfseries\var{13}};
\node[dot,fill=color7,text=white] (7) at(-120:\r) {\tiny\bfseries\var{7}};
\node[dot,fill=color12,text=white] (12) at (60+\ang:\s) {\tiny\bfseries\var{12}};
\node[dot,fill=color16,text=white] (16) at (60-\ang:\s) {\tiny\bfseries\var{16}};
\node[dot,fill=color8,text=white] (8)  at (180+\ang:\s) {\tiny\bfseries\var{8}};
\node[dot,fill=color9,text=white] (9) at (180-\ang:\s) {\tiny\bfseries\var{9}};
\node[dot,fill=color5,text=white] (5) at (-60+\ang:\s) {\tiny\bfseries\var{5}};
\node[dot,fill=color6,text=white] (6) at (-60-\ang:\s) {\tiny\bfseries\var{6}};
\draw[thick] (14) -- (7) (14) -- (15) (14) -- (13)
  (7) -- (8) (7) -- (6) (13) -- (9) (13) -- (12) (15) -- (5) (15) -- (16)
  (5) -- (6) (8) -- (9) (12) -- (16);
\end{tikzpicture}
}

\newcommandx{\hugesimplex}[8][1={0},2={0},3={0},4={0},5={0},6={0},7={0}]{%
\begin{tikzpicture}[scale=.5,baseline={([yshift=-5ex]current bounding box.center)}]
\def\r{1}
\def\s{2}
\def\t{2.5}
\def\u{3.5}
\definecolor{color1}{rgb}{1,1,1}
\definecolor{color2}{rgb}{1,1,1}
\definecolor{color3}{rgb}{1,1,1}
\definecolor{color4}{rgb}{1,1,1}
\definecolor{color5}{rgb}{1,1,1}
\definecolor{color6}{rgb}{1,1,1}
\definecolor{color7}{rgb}{1,1,1}
\definecolor{color8}{rgb}{1,1,1}
\definecolor{color9}{rgb}{1,1,1}
\definecolor{color10}{rgb}{1,1,1}
\definecolor{color11}{rgb}{1,1,1}
\definecolor{color12}{rgb}{1,1,1}
\definecolor{color13}{rgb}{1,1,1}
\definecolor{color14}{rgb}{1,1,1}
\definecolor{color15}{rgb}{1,1,1}
\definecolor{color16}{rgb}{1,1,1}
\definecolor{color17}{rgb}{1,1,1}
\definecolor{color18}{rgb}{1,1,1}
\definecolor{color19}{rgb}{1,1,1}
\definecolor{color20}{rgb}{1,1,1}
\definecolor{color#1}{rgb}{0,0,1}
\definecolor{color#2}{rgb}{0,0,1}
\definecolor{color#3}{rgb}{0,0,1}
\definecolor{color#4}{rgb}{0,0,1}
\definecolor{color#5}{rgb}{0,0,1}
\definecolor{color#6}{rgb}{0,0,1}
\definecolor{color#7}{rgb}{0,0,1}
\definevar{0}{}
\definevar{1}{}
\definevar{2}{}
\definevar{3}{}
\definevar{4}{}
\definevar{5}{}
\definevar{6}{}
\definevar{7}{}
\definevar{8}{}
\definevar{9}{}
\definevar{10}{}
\definevar{11}{}
\definevar{12}{}
\definevar{13}{}
\definevar{14}{}
\definevar{15}{}
\definevar{16}{}
\definevar{17}{}
\definevar{18}{}
\definevar{19}{}
\definevar{20}{}
\ifthenelse{#8>0}{
\redefinevar{#1}{1}
\redefinevar{#2}{2}
\redefinevar{#3}{3}
\redefinevar{#4}{4}
\redefinevar{#5}{5}
\redefinevar{#6}{6}
\redefinevar{#7}{7}
}{}
  \node[dot,fill=color14,text=white] (14) at (-162:1) {\tiny\bfseries\var{14}};
  \node[dot,fill=color15,text=white] (15) at (-90:1) {\tiny\bfseries\var{15}};
  \node[dot,fill=color16,text=white] (16) at (-18:1) {\tiny\bfseries\var{16}};
  \node[dot,fill=color12,text=white] (12) at (54:1) {\tiny\bfseries\var{12}};
  \node[dot,fill=color13,text=white] (13) at (126:1) {\tiny\bfseries\var{13}};
  \node[dot,fill=color7,text=white] (7) at (-162:\s) {\tiny\bfseries\var{7}};
  \node[dot,fill=color5,text=white] (5) at (-90:\s) {\tiny\bfseries\var{5}};
  \node[dot,fill=color17,text=white] (17) at (-18:\s) {\tiny\bfseries\var{17}};
  \node[dot,fill=color11,text=white] (11) at (54:\s) {\tiny\bfseries\var{11}};
  \node[dot,fill=color9,text=white] (9) at (126:\s) {\tiny\bfseries\var{9}};
  \node[dot,fill=color6,text=white] (6) at (-162+36:\t) {\tiny\bfseries\var{6}};
  \node[dot,fill=color4,text=white] (4) at (-90+36:\t) {\tiny\bfseries\var{4}};
  \node[dot,fill=color18,text=white] (18) at (-18+36:\t) {\tiny\bfseries\var{18}};
  \node[dot,fill=color10,text=white] (10) at (54+36:\t) {\tiny\bfseries\var{10}};
  \node[dot,fill=color8,text=white] (8) at (126+36:\t) {\tiny\bfseries\var{8}};
  \node[dot,fill=color2,text=white] (2) at (-162+36:\u) {\tiny\bfseries\var{2}};
  \node[dot,fill=color3,text=white] (3) at (-90+36:\u) {\tiny\bfseries\var{3}};
  \node[dot,fill=color19,text=white] (19) at (-18+36:\u) {\tiny\bfseries\var{19}};
  \node[dot,fill=color20,text=white] (20) at (54+36:\u) {\tiny\bfseries\var{20}};
  \node[dot,fill=color1,text=white] (1) at (126+36:\u) {\tiny\bfseries\var{1}};
\draw[thick] (14) -- (7) (14) -- (15) (14) -- (13)
  (7) -- (8) (7) -- (6) (13) -- (9) (13) -- (12) (15) -- (5) (15) -- (16)
  (5) -- (6) (8) -- (9) (12) -- (16)
  (1) -- (2) (1) -- (8) (1) -- (20) (2) -- (6) (2) -- (3) (3) -- (4) (3) -- (19)
  (4) -- (5) (4) -- (17) (10) -- (9) (10) -- (11) (10) -- (20) (11) -- (12) (11) -- (18)
  (16) -- (12) (16) -- (15) (16) -- (17) (17) -- (18) (18) -- (19) (19) -- (20);
\end{tikzpicture}
}

\newcommand{\newcomment}[4]{%
\newcounter{#2counter}
\expandafter\newcommand\csname #1\endcsname[1]{%
\refstepcounter{#2counter}%
{\color{#4}(#3\arabic{#2counter})}\marginpar{\scriptsize\raggedright\textbf{\color{#4}(#2 \arabic{#2counter}):} ##1}%
}}
\reversemarginpar

\definecolor{darkgreen}{rgb}{0,0.6,0}
\newcomment{sw}{Stefan}{S}{darkgreen}
\newcomment{ns}{Nada}{N}{blue}
\newcomment{ttm}{Thomas}{T}{red}

\begin{document}
\title{Vietoris--Rips complexes of platonic solids}
\date{\today}
\subjclass[2010]{Primary
                55N31; 
                Secondary %
                51M20, 
                57Q70, 
                05E45} 

\keywords{}

\author[N.~Saleh]{Nada Saleh}
\address{Mathematisches Institut, JLU Gießen, Arndtstr.\ 2, D-35392 Gießen, Germany}
\thanks{N.S.\ was supported through a DAAD scholarship.}
\email{saleh.nada@math.uni-giessen.de}

\author[T.~Titz Mite]{Thomas Titz Mite}
\address{Mathematisches Institut, JLU Gießen, Arndtstr.\ 2, D-35392 Gießen, Germany}
\email{thomas.titz-mite@math.uni-giessen.de}

\author[S.~Witzel]{Stefan Witzel}
\address{Mathematisches Institut, JLU Gießen, Arndtstr.\ 2, D-35392 Gießen, Germany}
\thanks{T.T.M\ and S.W.\ were supported through the DFG Heisenberg project WI 4079/6.}
\email{stefan.witzel@math.uni-giessen.de}

\begin{abstract}
  We determine the homotopy type of the Vietoris--Rips complexes of the (vertex sets of the) platonic solids. The most interesting case is that the Vietoris--Rips complex of the dodecahedron is a wedge of nine $3$-spheres when the parameter is between combinatorial distance $3$ and $4$.
\end{abstract}

\maketitle

The Vietoris--Rips complex $\VR_{r}(X)$ with parameter $r$ of a metric space $X$ provides a topological model for $X$ at scale $r$. It is used in topological data analysis at small scales to estimate the topological information pertaining to a dataset (a point cloud for instance), see for instance \cite{EdelsbrunnerHarer}. At large scales it is used in geometric group theory to analyse the topological finiteness properties of a group, see for instance \cite[Chapitre~4]{GhysdelaHarpe} and \cite{Alonso}. Usually one is interested in the rough behavior at some scale but in very structured situations one may obtain a complete picture, such as in \cite{AdamaszekAdams17} where all Vietoris--Rips complexes of the circle were determined. In this article we determine the Vietoris--Rips complexes of vertex sets of Platonic solids:

\begin{maintheorem}
Let $P$ be a platonic solid with vertex set $P^{(0)}$ and let $0 < \delta_1 < \ldots < \delta_k$ be the distances between pairs of elements of $P^{(0)}$. The homotopy type of $\VR_r(P^{(0)})$ is $\VR_r(P^{(0)}) = \emptyset$ if $r \le 0$, it is $\VR_r(P^{(0)}) \simeq *$ if $r > \delta_k$ and otherwise is as follows:
  \begin{enumerate}
    \item If $P$ is a tetrahedron then
      \[
        \VR_r(P^{(0)}) \simeq \begin{cases}
          \bigvee^3 S^0 & 0 < r \le \delta_1\\
        \end{cases}
      \]
    \item If $P$ is an octahedron then
      \[
        \VR_r(P^{(0)}) \simeq \begin{cases}
          \bigvee^5 S^0 & 0 < r \le \delta_1\\
          S^2 & \delta_1 < r \le \delta_2\\
        \end{cases}
      \]
    \item If $P$ is a cube then
      \[
        \VR_r(P^{(0)}) \simeq \begin{cases}
          \bigvee^7 S^0 & 0 < r \le \delta_1\\
          \bigvee^5 S^1 & \delta_1 < r \le \delta_2\\
          S^3 & \delta_2 < r \le \delta_3\\
        \end{cases}
      \]
    \item If $P$ is an icosahedron then
      \[
        \VR_r(P^{(0)}) \simeq \begin{cases}
          \bigvee^{11} S^0 & 0 < r \le \delta_1\\
          S^2 & \delta_1 < r \le \delta_2\\
          S^5 & \delta_2 < r \le \delta_3\\
        \end{cases}
      \]
    \item If $P$ is a dodecahedron then
      \[
        \VR_r(P^{(0)}) \simeq \begin{cases}
          \bigvee^{19} S^0 & 0 < r \le \delta_1\\
          \bigvee^{11} S^1 & \delta_1 < r \le \delta_2\\
          S^2 & \delta_2 < r \le \delta_3\\
          \bigvee^9 S^3 & \delta_3 < r \le \delta_4\\
          S^9 & \delta_4 < r \le \delta_5\\
        \end{cases}
      \]
  \end{enumerate}
\end{maintheorem}

Most of these statements are easy and likely known to experts except for the case of the dodecahedron in the range $(\delta_3,\delta_4]$. This case will occupy most of the article.

It is natural to try to extend this result to all regular polytopes. This is again easy for the simplices, the cross polytopes and the four-dimensional regular polytopes with the exception of the $120$-cell. For higher-dimensional cubes it does not remain true that the Vietoris--Rips complexes are wedges of spheres of a single dimension (as can be seen from homology computations), low dimensions have been studied in \cite{adamaszek}. The $120$-cell has combinatorial diameter $15$ leaving room for many potentially interesting Vietoris--Rips complexes. Other conceivable generalizations are to Archimedean solids or to (vertices of one type in) spherical buildings. 

The article is organized as follows. In Section~\ref{sec:basics} we introduce Vietoris--Rips complexes and basic combinatorial Morse theory and illustrate a difficulty pertaining to combinatorial Morse theory in the presence of symmetry. In Section~\ref{sec:simple} we prove the easy cases of the theorem: everything except for the dodecahedron in the range $\delta_3 < r \le \delta_4$. In Section~\ref{sec:warmup} we give an alternative prove of the cube in the range $\delta_2 < r \le \delta_3$ which on one hand serves as a preparation to the interesting case in the following section, but also generalizes to higher dimensional cubes. Finally, in Section~\ref{sec:main} we prove the main case of the theorem. GAP code for the article is available at \cite{code}.

\textbf{Acknowledgments.} We are grateful to Ralf Köhl for making us think about Vietoris--Rips complexes of spheres and to Henry Adams for suggesting this problem to us and generously sharing his thoughts.

\section{Basics}\label{sec:basics}

\subsection*{Vietoris--Rips complexes}

Let $X$ be a metric space, for instance a subset of $\R^n$ equipped with the standard metric. The \emph{Vietoris--Rips complex} $\VR_{r}(X)$ with parameter $r \ge 0$ is the simplicial complex with vertex set $X$ where a finite subset $F \subseteq X$ is the vertex set of a simplex if $\diam F < r$. We will also consider the complex $\VR_{\le r}(X)$ where simplices are required to satisfy $\diam F \le r$ instead. If $X$ is finite $\VR_r$ and $\VR_{\le r}$ determine each other in an obvious way. While $\VR_{r}$ is the complex usually considered, it will be convenient for us to use $\VR_{\le r}$ instead. The complexes $(\VR_{\le r}(X))_{r \ge 0}$ form an ascending sequence of complexes, leading up to the full simplex on $X$ (in the limit if $X$ has infinite diameter and reaching it has finite diameter). We will need to compare the complexes $\VR_{\le r}(X)$ and $\VR_{\le s}(X)$ for $r < s$ and use combinatorial Morse theory to do so.

\subsection*{Combinatorial Morse theory}
Combinatorial Morse theory was introduced by Forman \cite{forman} but we will refer to Kozlov's \cite{kozlov} formulation.

Let $C$ be a simplicial complex. If $\sigma,\tau \in C$ are simplices we write $\sigma < \tau$ to express that $\sigma$ is a face of $\tau$. If $\sigma$ has codimension one in $\tau$, we say that $\sigma$ is a \emph{facet} of $\tau$, that $\tau$ is a \emph{cofacet} of $\sigma$, and we write $\sigma \prec \tau$. A \emph{(partial) matching} on a simplicial complex $C$ is a relation $\match$ on the simplices $C$ subject to the conditions:
\begin{enumerate}
  \item for each pair $\sigma \match \tau$ of paired simplices we have $\sigma \prec \tau$, and\label{item:facet}
  \item for each simplex $\sigma$ there is at most one facet or cofacet (not both!) $\tau$ such that $\tau \match \sigma$ or $\sigma \match \tau$.\label{item:match}
\end{enumerate}
The simplices $\sigma$ for which a $\tau$ with $\tau \match \sigma$ or $\sigma \match \tau$ exists, are the ones that are \emph{matched} by the matching, the ones that are not are \emph{critical}.
If $\sigma \prec \tau$ but not $\sigma \match \tau$, we write $\tau \unmatch \sigma$. In this way a partial matching turns the Hasse diagram of (the face poset of) $C$ into a directed graph with the edge between $\sigma \prec \tau$ pointing up if the pair is matched and down if it is unmatched. The matching is \emph{acyclic} if this graph contains no directed cycle. Note that if such a cycle does exist, proving the matching non-acyclic, it necessarily has to be of the form
\[
  \sigma_1 \unmatch \tau_1 \match \ldots \sigma_k \unmatch \tau_k \match \sigma_1
\]
and, in particular, be restricted to adjacent dimensions, because there cannot be two consecutive arrows up by condition \eqref{item:match}. The result from combinatorial Morse theory that we will use is:

\begin{theorem}[{\cite[Theorem~10.9]{kozlov}}]\label{thm:morse}
  Let $C$ be a simplicial complex and let $D < C$ be a subcomplex. If there is an acyclic matching on $C$ such that the matched simplices are precisely those not in $D$, then $C$ deformation retracts onto $D$. More precisely the deformation retraction successively retracts $\sigma$ through $\tau$ onto the other facets of $\tau$ for each pair $\sigma \match \tau$.
\end{theorem}

Considering only simplices in adjacent dimensions makes combinatorial Morse theory very elegant but leads to nuisances in the presence of symmetry. To explain what we mean by that we consider a $3$-simplex with the indicated rotational symmetry, while ignoring its remaining symmetry:

\begin{center}
\begin{tikzpicture}[xshift=3cm,x={(1cm,-.1cm)},y={(.4cm,.18cm)},z={(0cm,1cm)}]
\pgfmathsetmacro{\sqrthalf}{sqrt(1/2)}
\coordinate (p) at (1,0,-\sqrthalf);
\coordinate (q) at (-1,0,-\sqrthalf);
\coordinate (r) at (0,1,\sqrthalf);
\coordinate (s) at (0,-1,\sqrthalf);
\draw (p) -- (q);
\draw (p) -- (r);
\draw (p) -- (s);
\draw[dashed] (q) -- (r);
\draw (q) -- (s);
\draw (r) -- (s);
\foreach \a in {p,q,r,s}
{
  \node[dot,fill=white] at (\a) {$\a$};
}
\tikzset{xyplane/.style={canvas is xy plane at z=#1}}
\begin{scope}[xyplane=-1.5]
\newcommand{\eps}{5}
  \draw[->] (0+\eps:1) arc (0+\eps:180-\eps:1);
  \draw[->] (180+\eps:1) arc (180+\eps:360-\eps:1);
\end{scope}
\end{tikzpicture}
\end{center}

Obviously the simplex can be deformation retracted to the bottom edge (which has codimension two) in an equivariant way, for instance by a straight-line retraction that moves $r$ to $q$ and $s$ to $p$. This collapse from the three-dimensional simplex to its one-dimensional face can be decomposed into collapses in adjacent dimensions but these collapses cannot be equivariant. The best we can do is to do one collapse that destroys symmetry followed by another that restores symmetry and so on. For instance we can collapse $prs \match pqrs$ destroying symmetry and then $rs \match qrs$ restoring symmetry. Then $pr \match pqr$ destroying symmetry and $qs \match pqs$ restoring it. Finally, $s \match ps$ destroying symmetry and $r \match qr$ restoring it.

\begin{table}
\begin{tabular}{l|C|C|C|}
  Polytope & \text{combinatorial} & \text{euclidean} & \text{spherical}\\
  \hline
  \multirow{1}{*}{Tetrahedron} & 1 & 2\sqrt{\frac{2}{3}} & \arccos\left(-\frac{1}{3}\right)\\
  \hline
  \multirow{3}{*}{Cube} & 1 & 2\sqrt{\frac{1}{3}} & \arccos\left(\frac{1}{3}\right)\\
                        & 2 & 2\sqrt{\frac{2}{3}} & \arccos\left(-\frac{1}{3}\right)\\
                        & 3 & 2 & \pi\\
  \hline
  \multirow{2}{*}{Octahedron} & 1 & \sqrt{2} & \pi/2\\
                              & 2 & 2 & \pi\\
  \hline
  \multirow{5}{*}{Dodecahedron} & 1 & \frac{2}{\varphi}\sqrt{\frac{1}{3}} & \arccos\left(\frac{2\varphi -1}{3}\right)\\
                                & 2 & 2\sqrt{\frac{1}{3}} & \arccos\left(\frac{1}{3}\right)\\
                                & 3 & 2\sqrt{\frac{2}{3}} & \arccos\left(-\frac{1}{3}\right)\\
                                & 4 & 2\varphi\sqrt{\frac{1}{3}} &\arccos\left(\frac{1 - 2\varphi}{3}\right)\\
                                & 5 & 2 & \pi \\
  \hline
  \multirow{3}{*}{Icosahedron} & 1 & 2 \sqrt{\frac{3 - \varphi}{5}}& \arccos\left(\frac{1}{\sqrt{5}}\right)\\
                               & 2 & 2 \sqrt{\frac{2 + \varphi}{5}} & \arccos\left(-\frac{1}{\sqrt{5}}\right)\\
                               & 3 & 2 & \pi\\
  \hline
\end{tabular}
\bigskip
\caption{Translation between combinatorial, euclidean, and spherical distance. Euclidean and spherical distance refer to the polytope inscribed in the unit sphere.}
\label{tab:distances}
\end{table}

\subsection*{Platonic solids} We may realize any platonic solid $P$ to have vertex set $P^{(0)}$ on the unit sphere $S^2 \subseteq \R^3$. This provides us with three possible distances on the set of vertices: the distance can be measured in the Euclidean space $\R^3$, on the sphere $S^2$, or combinatorially by counting the number of edges between them (Table~\ref{tab:distances}). Since these distances depend on each other monotonously, the same Vietoris--Rips complexes arise just for different parameters. It will be convenient to work with combinatorial distance, not only because it takes integer values but also because it allows us to work with the edge graph of the polytopes.

\section{Simple cases}\label{sec:simple}

Throughout we will consider a platonic solid $P$ and let $m,n,v,f,k$ denote its fundamental invariants as indicated in Table~\ref{tab:invariants}.
We want to determine the homotopy type of the Vietoris--Rips complexes $\VR_{\le r}(P^{(0)})$ with respect to combinatorial distance $r \in [0,\ldots,k]$ where $k$ the combinatorial diameter.

\begin{table}
  \centering
\begin{tabular}{l|ccccc|}
  Polytope & m & n & v & f & k\\
\hline
  Tetrahdron & 3 & 3 & 4 & 4 & 1\\
  Cube & 4 & 3 & 8 & 6 & 3\\
  Octahedron & 3 & 4 & 6 & 8 & 2\\
  Dodecahedron & 5 & 3 & 20 & 12 & 5\\
  Icosahedron & 3 & 5 & 12 & 20 & 3\\
  \hline
\end{tabular}
\bigskip
\caption{Basic combinatorics of the platonic solids. The numbers $m$ and $n$ are the vertices around a facet and the facets around a vertex, respectively. The numbers $v$ and $f$ are the numbers of vertices and facets respectively. The number $k$ is the diameter of the edge graph.}
\label{tab:invariants}
\end{table}

The first case is trivial.

\begin{lemma}\label{lem:vertices}
  If $0 \le r < 1$ then $\VR_{\le r}$ is the discrete set of $v$ vertices, so it is $\bigvee^{v-1} S^0$.\qed
\end{lemma}

\begin{lemma}\label{lem:cross_polytope}
  If $P = -P$ (so if $P$ is not the regular simplex) and if $k-1 \le r < k$ then $\VR_{\le r} \cong S^{v/2-1}$.
\end{lemma}

\begin{proof}
  For $x \in P^{(0)}$ the unique point at distance $k$ from $x$ is $-x$. So a face of the full simplex on $P^{(0)}$ is in $\VR_r$ unless it contains two opposite vertices. From this description we see that $\VR_r$ is (simplicially isomorphic to) $\bigast_v S^0$ and also the boundary of the $v$-dimensional cross polytope. Either description shows that $\VR_r$ is homeomorphic to $S^{v/2-1}$.
\end{proof}

The next case is again trivial.

\begin{lemma}
  If $m = 3$, $1 < k$ and $1 \le r < 2$ then $\VR_{\le r} = P^{(2)}$, the two-skeleton, so in particular $\VR_r \cong S^2$.\qed
\end{lemma}

The case $m \ne 3$ is slightly more interesting. Note that for these two polytopes the (combinatorial) diameter of a facet is $2$.

\begin{lemma}
  If $m > 3$ and $1 \le r < 2$ then $\VR_{\le r}= P^{(1)}$, so in particular $\VR_{\le r} \cong \bigvee_{f-1} S^1$.\qed
\end{lemma}

\begin{lemma}
  For the dodecahedron, if $2 \le r < 3$ then $\VR_{\le r} \simeq S^2$. More precisely every triangulation of $P^{(2)}$ without added vertices arises as a strong deformation retract of $\VR_{\le r}$.
\end{lemma}

\begin{proof}
  There are two kinds of maximal simplices in $\VR_{\le r}$: the first kind, call it $\sigma_x$, consists of a vertex $x$ of $P$ and its $m = 3$ neighbors, so in particular it is a $3$-simplex. The second kind, call it $\tau_F$, consists of the vertices of a facet $F$ of $P$, so it is an $(m-1)$-simplex. Matching $\sigma_x \setminus \{x\}$ with $\sigma_x$ for all $x$ describes a deformation retraction of $\VR_r$ onto the subcomplex of simplices each of which is contained some facet of $P$.

  It remains to see that $\tau_F$ deformation retracts to any triangulation of $F$. This is shown in the following lemma.
\end{proof}

When speaking of a triangulation of a polygon $F$ we shall always mean one whose vertex set is the vertex set of $F$. From the following lemma we only need the statement proved in the first paragraph (because we only care for $n$-gons with $n \le 5$), but we prove it in general.

\begin{lemma}\label{lem:polygon_retraction}
  Let $F \subseteq \R^2$ be a regular polygon, let $T$ be a triangulation of $F$, and let $\Delta$ be the full simplex on the vertex set of $F$. Regarding $T$ as a subcomplex of $\Delta$ there is a deformation retraction from $\Delta$ to $T$.
\end{lemma}

\begin{proof}
  First consider the special case where $T$ is the triangulation for which one fixed vertex $v$ is connected by an edge with each of the other vertices. We define a matching on the simplices of $\Delta$ not in $T$ by pairing a simplex $\sigma$ that contains $v$ with the complementary facet $\sigma \setminus v$. This matching is acyclic because if $\sigma \match \tau$ then $\tau$ contains $v$ and $\sigma$ does not while if $\tau \unmatch \sigma$ then both $\sigma$ and $\tau$ contain $v$. So a cycle would have to but cannot involve a simplex that does not contain $v$. Using Theorem~\ref{thm:morse} this shows that $\Delta$ deformation retracts onto $T$.

  It is well-known, \cite{wagner}, that any triangulation $T$ of $F$ can be transformed into any other triangulation $T'$ by a sequence of flips which consist in replacing a quadrangle $\gamma$ that is triangulated by a diagonal $\delta$ and triangles $\alpha, \beta$ by that quadrangle with its other triangulation which has diagonal $\delta'$ and triangles $\alpha',\beta'$:

  \begin{center}
  \begin{tikzpicture}[scale=1.3]
    \begin{scope}
      \node[dot] (a) at (5:.8) {};
      \node[dot] (b) at (80:1.1) {};
      \node[dot] (c) at (190:1.3) {};
      \node[dot] (d) at (-100:.8) {};
      \draw[thick] (a) -- (b) -- (c) -- (d) -- (a);
      \draw[thick] (a) -- (c);
      \node[fill=white] at ($(a)!.5!(c)$) {$\delta$};
      \coordinate (ab) at ($(a)!.5!(b)$);
      \coordinate (ad) at ($(a)!.6!(d)$);
      \node at ($(c)!.67!(ab)$) {$\alpha$};
      \node at ($(c)!.75!(ad)$) {$\beta$};
    \end{scope}
    \node at (1.8,0) {$\leadsto$};
    \begin{scope}[xshift=4cm]
      \node[dot] (a) at (5:.8) {};
      \node[dot] (b) at (80:1.1) {};
      \node[dot] (c) at (190:1.3) {};
      \node[dot] (d) at (-100:.8) {};
      \draw[thick] (a) -- (b) -- (c) -- (d) -- (a);
      \draw[thick] (b) -- (d);
      \node[fill=white] at ($(b)!.5!(d)$) {$\delta'$};
      \coordinate (cd) at ($(c)!.5!(d)$);
      \coordinate (ad) at ($(a)!.5!(d)$);
      \node at ($(b)!.67!(cd)$) {$\alpha'$};
      \node at ($(b)!.67!(ad)$) {$\beta'$};
    \end{scope}
  \end{tikzpicture}
  \end{center}

  Thus given a matching on simplices of $\Delta$ not in $T$ it suffices to remove $\alpha',\beta',\delta'$ from the matching and add $\alpha,\beta,\delta$ to it. In the original matching the $3$-simplex corresponding to $\gamma$ has to be matched with one of $\alpha'$ and $\beta'$ and $\delta'$ with the other of the two. Removing these two matches from the matching and adding instead the corresponding matches without primes leads to the required matching. The new matching is acyclic because in the resulting matching the edges of the quadrangle $\tau$ are unmatched as are all triangles except for the one matched with the simplex corresponding to $\tau$. Thus every maximal $\match$-$\unmatch$-sequence involving $\alpha,\beta,\delta$ ends in an unmatched edge or an unmatched triangle and, in particular, does not cycle.
\end{proof}

At this point we have proven the theorem except for the case of the dodecahedron with $\delta_3 < r \le \delta_4$.

\section{Warmup: the cube}\label{sec:warmup}

We could now proceed directly to the proof of the remaining case and the reader is free to do so in the next section. However, we choose to first give an alternative proof of the fact that if $P$ is the cube then the complex $X \defeq \VR_{\le 2}(P^{(0})$ is a $3$-sphere (this follows from Lemma~\ref{lem:cross_polytope} above). We do so for two reasons: first and most importantly, this proof follows the same strategy as the proof for the dodecahedron in the next section while being simpler, and second, it generalizes to higher-dimensional cubes.

The complex $X$ contains three types of maximal simplices of diameter $2$, each of which is a $3$-simplex:

\hfill
\begin{tikzpicture}[scale=1]
  \node[dot,fill=blue] (v) at (0,0) {};
  \node[dot,fill=blue] (a) at (-90:1) {};
  \node[dot,fill=blue] (b) at (30:1) {};
  \node[dot,fill=blue] (c) at (150:1) {};
  \draw[thick] (v) -- (a) (v) -- (b) (v) -- (c);
\end{tikzpicture}
\hfill
\begin{tikzpicture}[scale=1]
  \node[dot,fill=blue] (a) at ({sqrt(2)},0) {};
  \node[dot,fill=blue] (b) at ({sqrt(2)},{sqrt(2)}) {};
  \node[dot,fill=blue] (c) at (0,{sqrt(2)}) {};
  \node[dot,fill=blue] (d) at (0,0) {};
  \draw[thick] (a) -- (b) -- (c) -- (d) -- (a);
\end{tikzpicture}
\hfill
\begin{tikzpicture}[scale=1]
  \node[dot,fill=blue] (a) at ({sqrt(2)},0) {};
  \node[dot] (b) at ({sqrt(2)},{sqrt(2)}) {};
  \node[dot,fill=blue] (c) at (0,{sqrt(2)}) {};
  \node[dot] (d) at (0,0) {};
  \node[dot] (e) at ($({sqrt(2)},0)+(1/2,1/2)$) {};
  \node[dot,fill=blue] (f) at ($({sqrt(2)},{sqrt(2)})+(1/2,1/2)$) {};
  \node[dot] (g) at ($(0,{sqrt(2)})+(1/2,1/2)$) {};
  \node[dot,fill=blue] (h) at ($(0,0)+(1/2,1/2)$) {};
  \draw[thick] (a) -- (b) -- (c) -- (d) -- (a) (e) -- (f) -- (g) -- (h) -- (e) (a) -- (e) (b) -- (f) (c) -- (g) (d) -- (h);
\end{tikzpicture}
\hfill{}

Such a picture determines a class of simplices as follows: the graph can be embedded as a full subgraph into the edge graph of the cube and any such embedding gives a simplex, which consists of the images of the colored vertices. Thus the first kind consists of a vertex and its neighbors and there are eight of these, call them $\alpha_1,\ldots,\alpha_8$. The second consists of the vertices of a square and there are six of these, call them $\beta_1,\ldots,\beta_6$. The last kind finally consists of every other vertex of the cube and there are two of these, call them $\tau_1,\tau_2$.

Now we can understand $X = \VR_{\le 2}$ as follows. We first consider the subcomplex $Y \defeq X \setminus \{\tau_1,\tau_2\}$ obtained by removing the two simplices of the third kind. We define a matching on this subcomplex matching each $\alpha_i$ with the $2$-simplex consisting of the outer vertices (its unique facet any two of whose vertices have distance two). This defines a deformation retraction of $Y$ onto the subcomplex consisting of simplices that are contained in some $\beta_i$. This complex is the $2$-skeleton of the cube with each square fattened up to a $3$-simplex. In particular, it is easily seen to be homotopy equivalent to a $2$-sphere. More precisely, we can use Lemma~\ref{lem:polygon_retraction} to retract each $\beta_i$ to either of the two subdivisions of its corresponding square.

At this point we know that $X$ is the $2$-sphere $Y$ with two $3$-balls $\tau_1,\tau_2$ added to it. This is very plausibly but not necessarily a $3$-sphere. It remains to see how the balls are glued in, i.e. determine the maps $\partial \tau_i \cong S^2 \to S^2 \simeq Y$ up to homotopy. In order to do so, we need to trace how the facets of each $\tau_j$ are mapped under the deformation retraction defined by the matching. Each facet of $\tau_j$ is matched with some $\alpha_i$ and thus is retracted to the other three facets of $\alpha_i$. Thus the boundary of $\tau_1$ is retracted to some subdivision of the boundary of the $2$-skeleton of the cube --- and $\tau_2$ to a different one:
\begin{center}
\begin{tikzpicture}[scale=1.5]
  \node[dot,fill=blue] (a) at ({sqrt(2)},0) {};
  \node[dot,fill=blue] (c) at (0,{sqrt(2)}) {};
  \node[dot] (e) at ($({sqrt(2)},0)+(1/2,1/2)$) {};
  \node[dot,fill=blue] (f) at ($({sqrt(2)},{sqrt(2)})+(1/2,1/2)$) {};
  \node[dot] (g) at ($(0,{sqrt(2)})+(1/2,1/2)$) {};
  \node[dot,fill=blue] (h) at ($(0,0)+(1/2,1/2)$) {};
  \draw[thick] (a) -- (h) (c) -- (h) (f) -- (h);
  \node[dot] (d) at (0,0) {};
  \node[dot,fill=white] (b) at ({sqrt(2)},{sqrt(2)}) {};
  \draw[thick] (a) -- (b) -- (c) -- (d) -- (a) (e) -- (f) -- (g) -- (h) -- (e) (a) -- (e) (b) -- (f) (c) -- (g) (d) -- (h);
  \draw[thick] (a) -- (c) -- (f) -- (a);
  \begin{scope}[xshift=4cm]
  \node[dot,fill=blue] (a) at (-{sqrt(2)},0) {};
  \node[dot,fill=blue] (c) at (0,{sqrt(2)}) {};
  \node[dot] (e) at ($(-{sqrt(2)},0)+(1/2,1/2)$) {};
  \node[dot,fill=blue] (f) at ($(-{sqrt(2)},{sqrt(2)})+(1/2,1/2)$) {};
  \node[dot] (g) at ($(0,{sqrt(2)})+(1/2,1/2)$) {};
  \node[dot,fill=blue] (h) at ($(0,0)+(1/2,1/2)$) {};
  \draw[thick] (a) -- (h) (c) -- (h) (f) -- (h);
  \node[dot] (d) at (0,0) {};
  \node[dot,fill=white] (b) at (-{sqrt(2)},{sqrt(2)}) {};
  \draw[thick] (a) -- (b) -- (c) -- (d) -- (a) (e) -- (f) -- (g) -- (h) -- (e) (a) -- (e) (b) -- (f) (c) -- (g) (d) -- (h);
  \draw[thick] (a) -- (c) -- (f) -- (a);
  \end{scope}
\end{tikzpicture}
\end{center}
Since both of these subdivisions are retracts of $Y$ we see that the maps $\partial \tau_i \to Y$ are homotopy equivalences. It follows that $X$ is indeed a $3$-sphere.

The proof for the dodecahedron in the following section differs from this one in two aspects: first, there are ten rather than two $3$-simplices to remove and second, the deformation retraction is more complicated.

\subsection*{Higher-dimensional cubes at radius $\mathbf{2}$}

The above proof generalizes to $\VR_{\le 2}(C_n)$ where $C_n$ is an n-dimensional cube. The result has been shown in \cite{adamaszek} so we only sketch our proof. The maximal simplices of diameter $2$ are again of three types: $1$-balls around vertices, $2$-faces, and the third type above of which there are two within any $3$-face. As above one can remove the simplices of the third type and retract the remaining complex to a subdivision of the $2$-skeleton. It is clear that the $2$-skeleton is a wedge of spheres and that its homology is generated by the $2$-skeleta of $3$-faces. It is also clear (from the discussion of the $3$-cube above) that each $2$-skeleton of a $3$-face gets filled in by one of the simplices of the third kind. Thus $\VR_{\le 2}(C_n)$ is a wedge of $3$-spheres and it remains to determine their number which is twice the number of $3$-faces minus $b_2(C_n^{(2)})$.

The number $f_n^k$ of $k$-faces of the $n$-cube satisfies the relation $f_{n+1}^{k+1} = f_{n}^{k+1} + f_n^k$ with $f_0^0 = 1$ and $f_n^k = 0$ for $n < k$. This leads to the generating function
\[
  \alpha(x,y) = \sum_{n,k \ge 0} f_n^k x^n y^k = \frac{1}{1-(2+y)x}\text{.}
\]
Since $C_n^{(k)}$ is a wedge of $k$-spheres, $b_k(C_n^{(k)})$ is the reduced Euler characteristic signed appropriately:
\[
  \tilde{e}_n^k = (-1)^k (\chi(C_n)^{(k)} - 1) = b_k(C_n^{(k)})
\]
Its generating function is
\begin{align*}
  \beta(x,y) &= \sum_{n,k \ge 0} \tilde{e}_n^k x^n y^k\\
             &= \frac{1}{1+y}\alpha(x,y) - \frac{1}{(1+y)(1-x)} = \frac{x}{(1-x)(1-(2+y)x)}\text{.}
\end{align*}
Specializing to fixed $k$ we see that the generating functions in $n$ for $f_n^k$ and $b_k(C_n^{(k)})$ are
\begin{align*}
  \alpha_k(x) &= \frac{1}{k!}\frac{\partial^k}{\partial y^k} \alpha(x,0) = \frac{x^k}{(1-2x)^{k+1}} && \text{and}\\
  \beta_k(x) &= \frac{1}{k!}\frac{\partial^k}{\partial y^k} \beta(x,0) =  \frac{x^{k+1}}{(1-x)(1-2x)^{k+1}} \text{.}
\end{align*}
Since $b_3(\VR_{\le 2}(C_n)) = 2 f_n^3 - b_2(C_n^{(2)})$ it has generating function
\[
  2\alpha_3 - \beta_2 = \frac{x^{3}}{(1-x)(1-2x)^{4}}\text{.} 
\]
One notices, in particular, that this is $1/x \cdot \beta_3$ so that $b_3(\VR_{\le 2}(C_n)) = b_3(C_{n+1}^{(3)})$, possibly hinting at a different approach.

\section{The interesting case}\label{sec:main}

In this section we treat the case left after Section~\ref{sec:simple}: $P$ is the dodecahedron and $3 \le r < 4$, so we may just take $r = 3$. Let $\Gamma$ be the edge graph of $P$. We follow the strategy from the last section and start by identifying the critical simplices:

\begin{lemma}\label{lem:critical}
  The complex $\VR_{\le r}$ contains $10$ tetrahedra $\tau_1,\ldots,\tau_{10}$ such that any two of its vertices have distance $3$.\qed
\end{lemma}

These are well-known: under the action of the rotation group $H$ of the dodecahedron they fall into two orbits and the action on each orbit witnesses that the rotation group is $H \cong A_5$. Let $G = \Aut(P)$ denote the full symmetry group of $P$, so $G = H \times \{\pm 1\}$.
As in the previous section, we want to refer to $G$-orbits or \emph{types} of simplices: when drawing a graph $\Lambda$ with a set $B$ of vertices colored, it describes at the same time all simplices obtained as $\iota(B)$ were $\iota \colon \Lambda \to \Gamma$ is an embedding as a full subgraph. Taking the same graph and coloring a subset $B' \subseteq B$ we have specified for every simplex $\iota(B)$ a face $\iota(B')$. With this notation the simplices in Lemma~\ref{lem:critical} are those of the type

\[
  \hugesimplex[5][8][12][19]{0}\text{.}
\]

The main technical ingredient to our proof is:

\begin{figure}[htb]
  \centering
  \includegraphics[page=1]{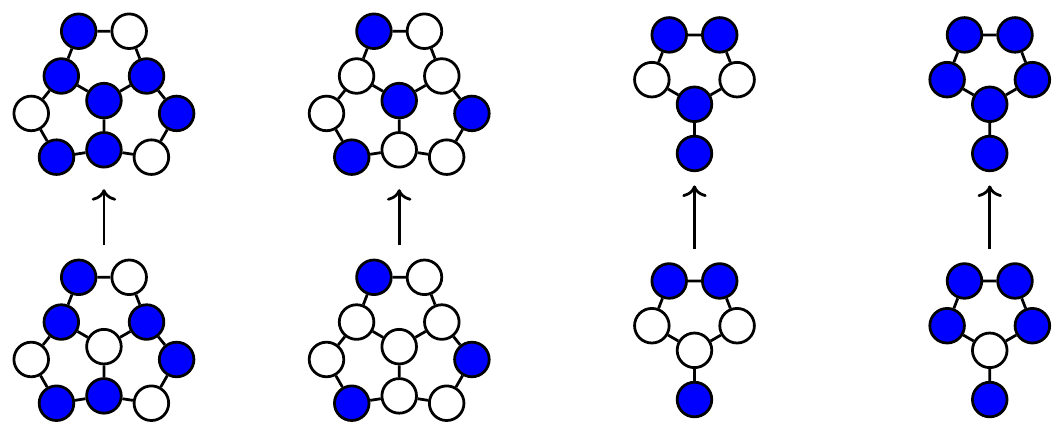}
\caption{Matched simplices with $C_3$ and $C_2$ symmetry.}
\label{fig:matching_symmetric}
\end{figure}

\begin{figure}
  \centering
  \includegraphics[page=2]{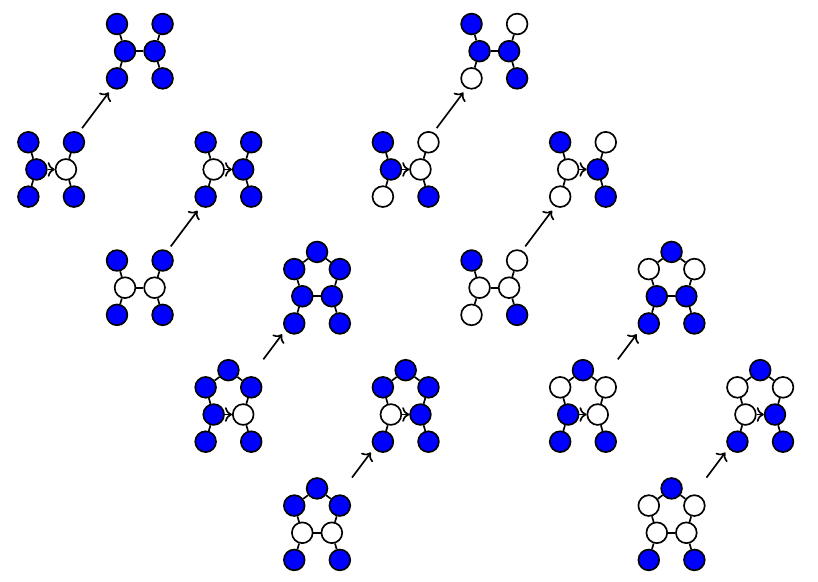}
\caption{Matched simplices with broken $C_2$ symmetry.}
\label{fig:matching_broken_symmetric}
\end{figure}

\begin{figure}
  \centering
  \includegraphics[page=3]{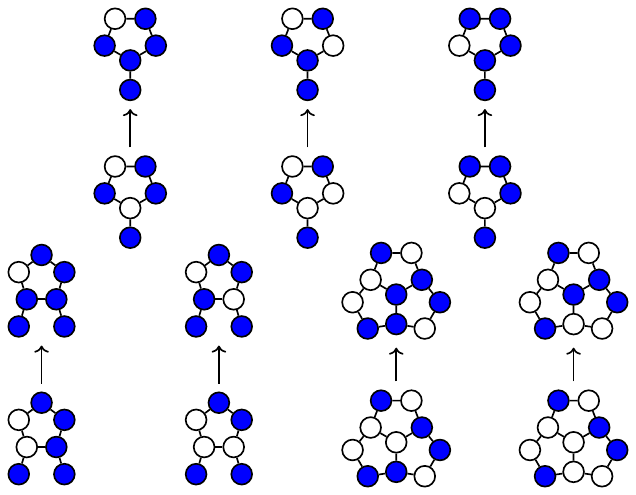}
\caption{Matched simplices with trivial stabilizer.}
\label{fig:matching_asymmetric}
\end{figure}

\begin{proposition}\label{prop:retraction_3_2}
  There is a strong deformation retraction of $\VR_{\le 3} \setminus \{\tau_1,\ldots,\tau_{10}\}$ to $\VR_{\le 2}$. More precisely there is a partial acyclic matching on the simplices of $\VR_{\le 3}$ of diameter $3$ whose only critical simplices are $\tau_1,\ldots,\tau_{10}$.
\end{proposition}

The matching will be $G$-equivariant up to the restrictions discussed in Section~\ref{sec:basics}. 
Since the matching cannot quite be equivariant, we make the following modification to our approach of specifying simplices: we fix once and for all a random orientation of the edges of $\Gamma$. Now we orient at most one edge of $\Lambda$ and require the embedding $\iota \colon \Lambda \to \Gamma$ to take that edge to an edge oriented in the same sense. Every unoriented edge of $\Lambda$ can be mapped to an edge with arbitrary orientation.

To illustrate the significance of this consider the three types of simplices

\begin{center}
  \smallsimplex[7][9][12][15][16][13][14]{0}{}, \quad
  \smallsimplex[7][9][12][15][16][13]{0}{->}, \quad and \quad
  \smallsimplex[7][9][12][15][16][14]{0}{->}.
\end{center}

We want to match the first type with the second. Note however, that the first type has a symmetry which the second and third type do not have. So without the edge orientation the first type would have the second as a facet in two ways but it cannot be matched with both. Thus the role of the edge orientation is to distinguish between (co)facets of simplices that have symmetry.

With these explanations on how to read pictures, the matching referred to in Proposition~\ref{prop:retraction_3_2} is the one given in Figures~\ref{fig:matching_symmetric},~\ref{fig:matching_broken_symmetric} and~\ref{fig:matching_asymmetric}. The heuristic to producing this matching was to preserve symmetry when possible as in Figure~\ref{fig:matching_symmetric}, and to otherwise destroy and recover it (as discussed in the introduction) as in Figure~\ref{fig:matching_broken_symmetric}.

\begin{proof}[Proof of Proposition~\ref{prop:retraction_3_2}]
  To prove the proposition we need to verify that every simplex of diameter $3$ in $\VR_{\le 3}(\Gamma)$ except for $\tau_1,\ldots,\tau_{10}$ appears in Figures~\ref{fig:matching_symmetric},~\ref{fig:matching_broken_symmetric} and~\ref{fig:matching_asymmetric}. This is tedious but easy and we leave it to the reader (or their computer).

  It remains to see that the matching is acyclic. This is done by showing that the relations $\match$ and $\unmatch$ on simplices of $\VR_{\le 3}(\Gamma)$ of adjacent dimensions $d,d+1$ can be refined to a partial order. More precisely, a potential $\match$-$\unmatch$ cycle in dimensions $d,d+1$ would involve only simplices that are part of a matched pair in these dimensions. It thus suffices to show that the simplices that are part of such a matching can be ordered. Such orders are depicted in Figures~\ref{fig:acyc_5_6} to~\ref{fig:acyc_1_2}.
\end{proof}

\begin{figure}
  \begin{center}
    \includegraphics[page=4]{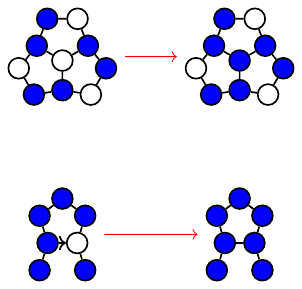}
  \end{center}
  \caption{Acyclicity of the matching in dimensions $5$ and $6$. A red arrow from $\sigma$ to $\tau$ means that $\sigma \match \tau$. All relations $\match$ are shown and there are no relations $\unmatch$ among the simplices. The point is that every arrow is pointing right.}
  \label{fig:acyc_5_6}
\end{figure}

\begin{figure}
  \begin{center}
    \includegraphics[page=5]{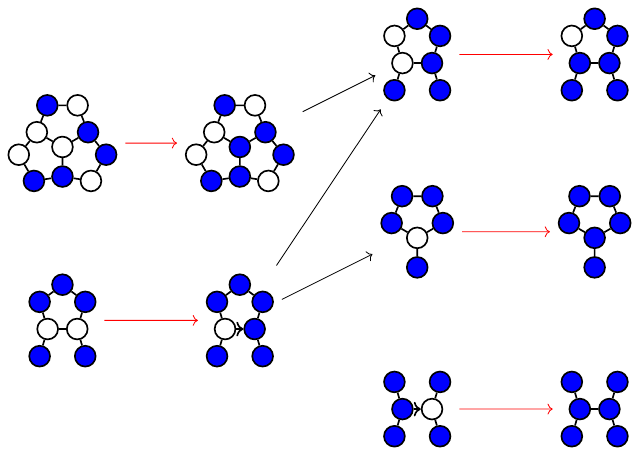}
  \end{center}
  \caption{Acyclicity of the matching in dimensions $4$ and $5$. A red arrow from $\sigma$ to $\tau$ means that $\sigma \match \tau$. A black arrow from $\tau$ to $\sigma$ means that $\tau \unmatch \sigma$. All relations $\match$ are shown and all relations $\unmatch$ among the simplices as well. The point is that every arrow is pointing right.}
  \label{fig:acyc_4_5}
\end{figure}

\begin{figure}
  \begin{center}
    \includegraphics[page=6]{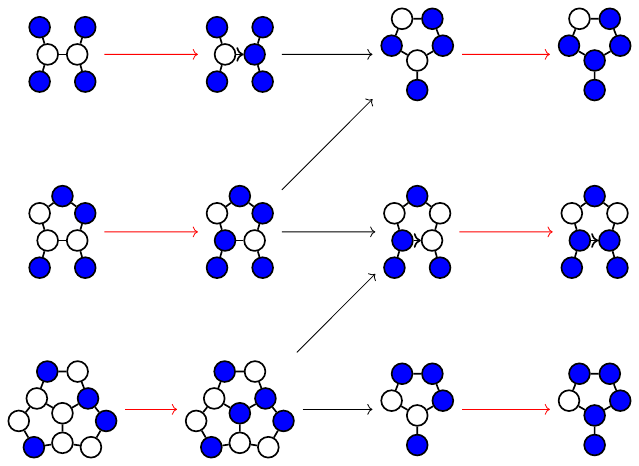}
  \end{center}
  \caption{Acyclicity of the matching in dimensions $3$ and $4$. The meaning is as in Figure~\ref{fig:acyc_4_5}.}
  \label{fig:acyc_3_4}
\end{figure}

\begin{figure}
  \begin{center}
    \includegraphics[page=7]{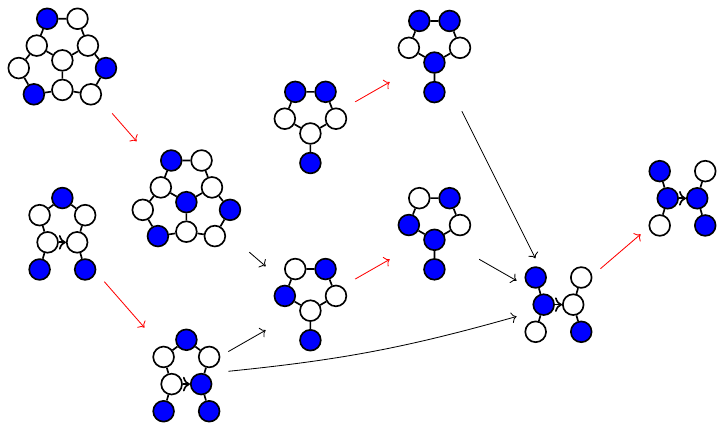}
  \end{center}
  \caption{Acyclicity of the matching in dimensions $2$ and $3$. The meaning is as in Figure~\ref{fig:acyc_4_5}.}
  \label{fig:acyc_2_3}
\end{figure}

\begin{figure}
  \begin{center}
    \includegraphics[page=8]{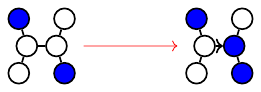}
  \end{center}
  \caption{Acyclicity of the matching in dimensions $1$ and $2$. The meaning is as in Figure~\ref{fig:acyc_4_5}.}
  \label{fig:acyc_1_2}
\end{figure}
To prove our theorem it remains to show that the map
\[
 \partial \tau_i \cong S^2 \to S^2 \simeq \VR_{\le 3} \setminus \{\tau_1,\ldots,\tau_{10}\}
\]
along which $\tau_i$ is glued in is a homotopy equivalence. To study this map is a purely algorithmic matter: we just have to homotope it along the deformation retraction from Proposition~\ref{prop:retraction_3_2}. In order to facilitate the exposition we make use of Hurewicz's theorem:

\begin{theorem}[{\cite[Theorem~4.32]{hatcher}}]
  If a space $X$ is $1$-connected then the map $\pi_2(X) \to H_2(X)$ is an isomorphism.
\end{theorem}

Therefore it suffices to show:

\begin{proposition}\label{prop:trace_retraction}
  For each $i$ the homology class $[\partial \tau_i]$ generates
  \[H_2(\VR_{\le 3} \setminus \{\tau_1,\ldots,\tau_{10}\}) \cong H_2(\VR_{\le 2}) \cong H_2(S^2) \cong \Z\text{.}\]
\end{proposition}

In order to compute integral homology we need to deal with oriented simplices. We do this by adding numbers to the colored blue vertices $B$ of our defining graphs $\Lambda$. As usual permuting the numbers of a simplex $\sigma$ by an even permutation gives the same simplex, while permuting them by an odd permutation gives the  simplex with opposite orientation, i.e. $-\sigma$. Our sign convention for boundaries is that
\[
  \partial [1,\ldots,n] = \sum_{i=1}^n (-1)^i [1,\ldots, \hat{i}, \ldots, n]\text{.}
\]

\begin{proof}[Proof of Proposition~\ref{prop:trace_retraction}]
  The facets of the simplices $\tau_1,\ldots,\tau_{10}$ are of type
  \[
    \varphi = \simplex[6][16][9]{1}
  \]
  so we need to know how they are mapped (homologically) under the deformation retraction. Our matching matches them with the simplices
  \[
    \psi = \simplex[6][16][9][14]{1}.
  \]
  In order to determine the image of $\varphi$ under the retraction we compute the boundary of $\psi$ as
  \[
  \partial \psi = - \simplex[16][9][14]{1} + \simplex[6][9][14]{1}  - \simplex[6][16][14]{1} + \simplex[6][16][9]{1}\text{.}
  \]
  Thus under the first step of the retraction $\varphi$ is taken to 
  \[
    \simplex[16][9][14]{1} + \simplex[6][14][9]{1}  + \simplex[6][16][14]{1}\text{.}
  \]
  Each summand is again matched with a cofacet and we compute its boundary as
  \[
    \partial \simplex[16][9][14][13]{1} = - \simplex[9][14][13]{1} + \simplex[16][14][13]{1} - \simplex[16][9][13]{1} +\simplex[16][9][14]{1}
  \]
  so the image of $\varphi$ after two retraction steps is
  \begin{align*}
                     \omega &= \simplex[9][14][13]{1} + \simplex[16][14][15]{1} + \simplex[6][14][7]{1}\\
                     & + \simplex[16][13][14]{1} + \simplex[6][15][14]{1} + \simplex[9][7][14]{1} \\
                     &+ \simplex[16][9][13]{1}+  \simplex[6][16][15]{1}+ \simplex[9][6][7]{1}
  \end{align*}
  The summands in the first two rows lie in $\VR_{\le 2}(\Gamma)$. For the summands in the last row
  \[
    \rho = \tinysimplex[7][12][14]{1}{->} \quad \text{or} \quad \rho' = \tinysimplex[7][12][13]{1}{->}
  \]
  the retraction image depends on edge orientations in $\Gamma$. While $\rho$ is matched with a facet and therefore does not contribute to second homology, $\rho'$ is matched with a cofacet whose boundary is
  \[
   \partial \tinysimplex[7][12][13][14]{1}{->} = \tinysimplex[13][12][14]{1}{->} + \tinysimplex[7][13][14]{1}{->} + \tinysimplex[12][7][14]{1}{->} + \tinysimplex[7][12][13]{1}{->}\text{.}
  \]
  What we retain from this is that the sum $\rho + \rho'$ is retracted to
  \[
    \kappa + \kappa' \defeq \tinysimplex[12][13][14]{1}{} + \tinysimplex[7][14][13]{1}{}
  \]
  
  It remains to make one observation. Each simplex $\tau_i$ has four facets of type $\varphi$. While we do not know whether a single simplex of type $\varphi$ retracts through simplices of type $\rho'$ or $\rho$ and thus has $\kappa$ in its retraction image or not, the sum of the four facets will have $\rho + \rho'$ in its retraction image and therefore $\kappa + \kappa'$ as well.

  Thus in summary the boundary of $\tau_i$ retracts to the sum of four chains of type
  \begin{align*}
                     \omega' &= \simplex[9][14][13]{1} + \simplex[16][14][15]{1} + \simplex[6][14][7]{1}\\
                     & + \simplex[16][13][14]{1} + \simplex[6][15][14]{1} + \simplex[9][7][14]{1} \\
                     & + \simplex[16][12][13]{1}+  \simplex[6][5][15]{1}+ \simplex[9][8][7]{1}
  \end{align*}
  each of which corresponds to the subdivision
  \begin{center}
\begin{tikzpicture}[scale=1,rotate=30,baseline={([yshift=-.8ex]current bounding box.center)}]
\def\r{1}
\def\s{1.5}
\def\ang{20}
\node[dot,minimum size=4pt] (14) at (0,0) {};
\node[dot,minimum size=4pt] (15) at (0:\r) {};
\node[dot,minimum size=4pt] (13) at (120:\r) {};
\node[dot,minimum size=4pt] (7) at(-120:\r) {};
\node[dot,minimum size=4pt] (12) at (60+\ang:\s) {};
\node[dot,minimum size=4pt] (16) at (60-\ang:\s) {};
\node[dot,minimum size=4pt] (8)  at (180+\ang:\s) {};
\node[dot,minimum size=4pt] (9) at (180-\ang:\s) {};
\node[dot,minimum size=4pt] (5) at (-60+\ang:\s) {};
\node[dot,minimum size=4pt] (6) at (-60-\ang:\s) {};
\draw[thick] (14) -- (7) (14) -- (15) (14) -- (13)
  (7) -- (8) (7) -- (6) (13) -- (9) (13) -- (12) (15) -- (5) (15) -- (16)
  (5) -- (6) (8) -- (9) (12) -- (16);
\draw[thick] (14) -- (16) (14) -- (9) (14) -- (6)
  (7) -- (9) (6) -- (15) (16) -- (13);
\end{tikzpicture}
\end{center}
  of the corresponding subcomplex of $P^{(2)}$. The sum represents a homology class generating $H_2(P^{(2)})$.
\end{proof}

\begin{remark}
  From the proof one sees that $H_3(\VR_{\le 3}(P^{(2)}))$ is equivariantly isomorphic to the $\Z G$-module
  \[
    \ker( \Z \tau_1 \oplus \ldots \oplus \Z \tau_{10} \stackrel{\varphi}{\to} \Z)
  \]
  where $\varphi (\tau_i) = 1$ for all $i$. This allows to read off the $G$-module structure, at least with field coefficients. When numbered appropriately, $H \cong A_5$ acts on each of $\tau_1,\ldots,\tau_5$ and $\tau_6,\ldots,\tau_{10}$ by its natural permutation representation while the $\{\pm 1\}$ factor swaps $\tau_i$ and $\tau_{i+5}$ for $1 \le i \le 5$.

  Using representation notation we let $1 \oplus \nu$ be the natural representation of $A_5$ where $1$ is the trivial representation and $\nu$ is the ($4$-dimensional) standard representation. We also let $1$ and $\sigma$ be the trivial and the sign representation of $\{\pm 1\}$ respectively. We denote by the same symbols the representations of $G$ obtained by composing with the quotient homomorphism (which makes sense because both trivial representations give the trivial representation). Then our discussion above shows that $\Q \tau_1 \oplus \ldots \oplus \Q \tau_{10}$ is the representation $(1 \oplus \nu) \otimes (1 \oplus \sigma) = 1 \oplus \sigma \oplus \nu \oplus (\nu \otimes \sigma)$. It follows that $H_3(\VR_{\le 3}(P^{(2)});\Q)$ as a $G$-representation is $\sigma \oplus \nu \oplus (\nu \otimes \sigma)$.
\end{remark}


\providecommand{\bysame}{\leavevmode\hbox to3em{\hrulefill}\thinspace}
\providecommand{\MR}{\relax\ifhmode\unskip\space\fi MR }
\providecommand{\MRhref}[2]{%
  \href{http://www.ams.org/mathscinet-getitem?mr=#1}{#2}
}
\providecommand{\href}[2]{#2}

\end{document}